\documentclass[12pt,reqno]{amsart}
\usepackage{amsmath}
\newtheorem{defi}{Definition}
\newcommand{\brdef}{\begin{defi}}
\newcommand{\erdef}{\end{defi}}
\newtheorem{prop}{Proposition}
\newcommand{\bprop}{\begin{prop}}
\newcommand{\eprop}{\end{prop}}
\newtheorem{thm}{Theorem}
\newcommand{\bth}{\begin{thm}}
\newcommand{\eth}{\end{thm}}
\newtheorem{cor}{Corollary}
\newcommand{\bcor}{\begin{cor}}
\newcommand{\ecor}{\end{cor}}

\newtheorem{lem}{Lemma}
\newcommand{\ble}{\begin{lem}}
\newcommand{\ele}{\end{lem}}
\newtheorem{exm}{Example}
\newcommand{\bex}{\begin{exm}}
\newcommand{\eex}{\end{exm}}

 \huge

\usepackage{hyperref}

\numberwithin{equation}{section}

\begin{document}
\title{Ricci flow on modified Riemann Extensions}
\author[ H. G. Nagaraja and Harish D.]{H. G. Nagaraja and
 Harish D.}
\address[H. G. Nagaraja, Harish D.]{Department of Mathematics, Bangalore University, Central College Campus, Bengaluru-560001, INDIA}
\email{hgnraj@yahoo.com}
\email{itsme.harishd@gmail.com}
\footnote{This work is supported by CSIR 09/039(0106)2012-EMR-I}
\maketitle
\begin{abstract}
We study the properties of Modified Riemann extensions evolving under Ricci flow. We obtain the necessary and sufficient condition for modified Riemann extension under Ricci flow to stay as modified Riemann extension. We also discuss the properties of the curvature tensors under Ricci flow.
\end{abstract}

\section{introduction}
Ricci flow and the evolution equations of the Riemannian curvature tensor were initially introduced by  Richard Hamilton\cite{a} and was later studied to a large extent by Perelman\cite{i},\cite{j},\cite{k}, Cao, Zhu\cite{h}, John Morgan, Gang Tian\cite{l} and others. Indeed, the theory of Ricci flow has  been used to prove the geometrization and Poincare conjectures\cite{h}. However not much work has been  done on Ricci flows on Modified Riemann extensions. The Ricci flow equation is the evolution equation $\dfrac{\partial g_{il}}{\partial t}=-2R_{il}$ where $g_{il}$ and $R_{il}$ are the metric and Ricci tensor components respectively.  As flow progresses the metric changes and hence the properties related to it. \par

Patterson and Walker\cite{c} have defined Riemann extensions and showed how a Riemannian structure can be given to the $2n$ dimensional tangent bundle of an $n-$ dimensional manifold with given non-Riemannian structure. Riemann extension  is an embedding of a manifold $M$  in a manifold $M'$, the embedding being carried out in such a way that the geodesic equations are preserved up to the base space. The Riemann extension of Riemannian or non-Riemannian spaces can be constructed with the help of the Christoffel coefficients $\Gamma^i_{jk}$ of corresponding Riemann space or with connection coefficients $\Pi^i_{jk}$ in the case of the space of affine connection\cite{g}. The theory of Riemann extensions has been extensively studied by Afifi\cite{h}. Though the Riemann extensions is rich in geometry, here in our discussions, the modified Riemann extensions  fit naturally in to the frame work. Modified Riemann extensions were introduced in \cite{m} and \cite{n} and their properties we list briefly in the next section.

Here in this paper we discuss some interesting properties satisfied by curvature tensors under the influence of Ricci flow on modified Riemann extensions. We give a brief introduction to modified Riemann Extensions\cite{o} in section 2. In Section 3 we find the rate of change of concircular, conharmonic and conformal curvature tensors under Ricci flow. Ricci flow on modified Riemann extensions are discussed in Section 4.\par

\section{preliminaries}
Let $(M,g)$ be a n-dimensional Riemannian manifold. Then Ricci flow is the evolution of the metric given by 
\begin{equation}\label{ric}
\dfrac{\partial g_{il}}{\partial t}=-2R_{il},
\end{equation}
where $g_{il}$ is the metric component and $R_{il}$ is the component of the Ricci curvature tensor.

 For a time dependent metric under Ricci flow, the evolution equations for Riemann curvature tensor, Ricci tensor and scalar curvature are given by\cite{a}, 
\begin{equation}
\begin{split}
\frac{\partial R_{ijkl}}{\partial t}=&\triangle R_{ijkl}+2(B_{ijkl}-B_{ijlk}-B_{iljk}+B_{ikjl})\\&-g^{pq}(R_{pjkl}R_{qi}+R_{ipkl}R_{qj}+R_{ijpl}R_{qk}+R_{ijkp}R_{ql}),
\end{split}
\end{equation}
\begin{equation}\label{init}
\frac{\partial R_{ij}}{\partial t}=\triangle R_{ij}+2g^{pr}g^{qs}R_{piqj}R_{rs}-2g^{pq}R_{pi}R_{qj}
\end{equation}
and 
\begin{equation}
\frac{\partial R}{\partial t}=\triangle R+2g^{ij}g^{kl}R_{ik}R_{jl},
\end{equation}
where  $B_{ijkl}=g^{pr}g^{qs}R_{piqj}R_{rksl}$ and $R_{ijkl}$, $R_{ij}$, $R$ are the Riemannian curvature tensor, Ricci tensor and scalar curvature respectively.   \\\par
Let $\nabla$ be a torsion-free affine connection of $M$.
The modified Riemann extension of $(M, \nabla)$ is the cotangent bundle $T^*M$ equipped with a metric $\bar{g}$ whose local components given by
\\
$\bar{g}_{ij}=-2 \omega_l \Gamma^l_{ij}+c_{ij},\,\, \bar{g}_{ij^*}=\delta_i^j, \,\,\bar{g}_{i^*j}=\delta_i^j \,\,\textrm{and}\,\, \bar{g}_{i^{*}j^{*}}=0$. \\ \\
 The contravariant components are \\
 $\bar{g}^{ij}=0,\,\, \bar{g}^{ij^*}=\delta_i^j,\,\, \bar{g}^{i^{*}j}=\delta_i^j  \,\,\textrm{and}\,\, \bar{g}^{i^*j^*}=2 \omega_l \Gamma^l_{ij}-c_{ij}$\\ for $i,j$ ranging from $1$ to $n$ and $i^*,j^*$ ranging from $n+1$ to $2n$,\\
 where $\omega_l$ are extended coordinates and $c_{ij}$ is a $(0,2)$ tensor on $M$.\\

The Christophel symbols  are given by,
\begin{equation}
\Gamma^k_{ij}=\frac{1}{2}g^{kl}\left(\frac{\partial}{\partial x^i}g_{jk}+\frac{\partial}{\partial x^j}g_{ik}-\frac{\partial}{\partial x^k}g_{ij}\right).
\end{equation}
We note these results for the extended space,
$\bar{\Gamma}^k_{ij}=\Gamma^k_{ij}$,\\
$\bar{\Gamma}^{k^{*}}_{ij}=\omega_l R^l_kji+\frac{1}{2}(\nabla_i c_{jk}+\nabla_j c_{ik}-\nabla_h c_{ij})$,\\
$\bar{\Gamma}^k_{i^* j}=0, \,\, \bar{\Gamma}^{k^*}_{i^* j}=-\Gamma^i_{jk}, \,\, \bar{\Gamma}^k_{i^* j^*}=0, \,\, \bar{\Gamma}^{k^*}_{i^* j^*}=0$\\
The components of the Riemann curvature tensor of the extended space is given by \\
\begin{equation}
R^l_{ijk}=\frac{\partial}{\partial x^i}\Gamma^l_{jk}-\frac{\partial}{\partial x^j}\Gamma^l_{ik}+\Gamma^l_{im}\Gamma^m_{jk}-\Gamma^l_{jm}\Gamma^m_{ik}.
\end{equation}
For the extended space ,\\
$\bar{R}^i_{jkl}=R^i_{jkl}$,\\
$\bar{R}^{i^{*}}_{jkl}=\omega_a(\nabla_j R^a_{ilk}-\nabla_k R^a_{ilj})+\frac{1}{2}[\nabla_j(\nabla_l c_{ki}-\nabla_i c_{kl})-\nabla_k(\nabla_l c_{ji}-\nabla_i c_{jl})-R^m_{jkl}c_{mi}-R^m_{jki}c_{lm}]$,
$\bar{R}^{i^{*}}_{{j^*}kl}=R^j_{ilk}$,\;
$\bar{R}^{i^{*}}_{j{k^*}l}=-R^k_{ilj}$,\;
 and $\bar{R}^{i^{*}}_{jk{l^*}}=R^l_{kji}$.\\
The others are zero. $i^*,j^*, k^*,l^*$ ranges from $n+1$ to $2n$.
We lower the index in the middle position, to get
\begin{equation}
R_{ijkl}=g_{mk}R^m_{ijl}.
\end{equation}
It may be noted by simple calculation that 
$\bar{R}_{i^*jk^*l}=0$ which we require later on.
Further,
$\bar{R}_{ij}=R_{ij}+R_{ji}$,\;$\bar{R}_{i^*j}=0$ and $\bar{R}_{i^*j^*}=0$.

\section{evolution}
The Ricci flow is given by equation\eqref{ric}. As time progresses the metric evolves and hence the properties depending on the metric change. Under Ricci flow, the rate of change of conformal curvature depends on the difference of conharmonic and Riemannian curvature tensors.
The concircular curvature tensor is given by, 
\begin{equation}\label{conf}
C_{ijkl}=R_{ijkl}-\frac{R}{n(n-1)}[g_{il}g_{jk}-g_{jl}g_{ik}].
\end{equation}
The conharmonic curvature tensor is given by
\begin{equation}\label{conh}
L_{ijkl}=R_{ijkl}-\frac{1}{n-2}[g_{jk}R_{il}+g_{il}R_{jk}-g_{ik}R_{jl}-g_{jl}R_{ik}].
\end{equation}
Under Ricci flow, we give a relation between conformal tensor and conharmonic tensor.

\begin{thm}
For a manifold with non zero scalar curvature under Ricci flow, the rate of change of concircular tensor is related to conharmonic tensor by
\begin{equation}
\frac{\partial }{\partial t}(\frac{C_{ijkl}-R_{ijkl}}{R})=\frac{2(n-2)}{n(n-1)}[R_{ijkl}-L_{ijkl}].
\end{equation}
\end{thm}

\begin{proof} Differentiating \eqref{conf}  we get,
\begin{equation}
\begin{split}
\frac{\partial C_{ijkl}}{\partial t}=&\frac{\partial R_{ijkl}}{\partial t}-\frac{1}{n(n-1)}[g_{il}g_{jk}-g_{jl}g_{ik}]\frac{\partial R}{\partial t}\\&-\frac{R}{n(n-1)}\left[\frac{\partial g_{il}}{\partial t}g_{jk}+\frac{\partial g_{jk}}{\partial t}g_{il}-\frac{\partial g_{ik}}{\partial t}g_{jl}-\frac{\partial g_{jl}}{\partial t}g_{ik}\right].
\end{split}
\end{equation}

\begin{equation}\label{star}
\begin{split}
\frac{\partial C_{ijkl}}{\partial t}-\frac{\partial R_{ijkl}}{\partial t}=&-\frac{1}{n(n-1)}[g_{il}g_{jk}-g_{jl}g_{ik}]\frac{\partial R}{\partial t}\\&+\frac{R}{n(n-1)}\left[2R_{il}g_{jk}+2R_{jk}g_{il}-2R_{ik}g_{jl}-2R_{jl}g_{ik}\right].
\end{split}
\end{equation}
But 
\begin{equation}\label{1}
[g_{il}g_{jk}-g_{jl}g_{ik}]=\frac{n(n-1)}{R}(R_{ijkl}-C_{ijkl})
\end{equation}
and 
\begin{equation}\label{2}
R_{il}g_{jk}+R_{jk}g_{il}-R_{ik}g_{jl}-R_{jl}g_{ik}=(n-2)(R_{ijkl}-L_{ijkl}).
\end{equation}

Substituting \eqref{1} and \eqref{2} in \eqref{star} we get

\begin{equation}\label{star}
\begin{split}
\frac{\partial C_{ijkl}}{\partial t}-\frac{\partial R_{ijkl}}{\partial t}=-\frac{1}{R}(R_{ijkl}-C_{ijkl})\frac{\partial R}{\partial t}+\frac{2(n-2)R}{n(n-1)}(R_{ijkl}-L_{ijkl}).
\end{split}
\end{equation}

\begin{equation}\label{star}
\begin{split}
R\frac{\partial }{\partial t}(C_{ijkl}-R_{ijkl})-(C_{ijkl}-R_{ijkl})\frac{\partial R}{\partial t}=\frac{2(n-2)R^2}{n(n-1)}(R_{ijkl}-L_{ijkl}).
\end{split}
\end{equation}

Therefore,
\begin{equation}\label{eq4}
\frac{\partial }{\partial t}(\frac{C_{ijkl}-R_{ijkl}}{R})=\frac{2(n-2)}{n(n-1)}[R_{ijkl}-L_{ijkl}].
\end{equation}

\end{proof}

\begin{exm}
\rm{
Let $M$ be a manifold with a space of constant curvature with $K=\frac{1}{1-n}$. Then evolution of the metric under Ricci flow is given by $g_{ij}(t)=g_{ij}(0)e^{-2t}$ and $R_{ijkl}(t)=R_{ijkl}(0)e^{-4t}$.
Further, $C_{ijkl}-R_{ijkl}=-\frac{R}{n} R_{ijkl}$ and $L_{ijkl}-R_{ijkl}=\frac{2(n-1)}{n-2}R_{ijkl}$.
Substituting this in equation \eqref{eq4} the above result is verified.}
\end{exm}

For a Riemannian manifold the Weyl conformal tensor is given by, 
\begin{equation}\label{weyl}
\begin{split}
W_{ijkl}=&R_{ijkl}-\frac{1}{n-2}(g_{jk}R_{il}-g_{ik}R_{jl}+g_{il}R_{jk}-g_{jl}R_{ik})\\&+\frac{R}{(n-1)(n-2)}(g_{il}g_{jk}-g_{jl}g_{ik}).
\end{split}
\end{equation}

Equations \eqref{conf},\eqref{conh} and \eqref{weyl} can be combined to form,
\begin{equation}\label{thm3}
(W_{ijkl}-L_{ijkl})=-\frac{n}{n-2}(C_{ijkl}-R_{ijkl}).
\end{equation}

\begin{thm} \label{thm2}
 For a n-manifold under Ricci flow,
\begin{equation}
\frac{\partial }{\partial t}\frac{(W_{ijkl}-L_{ijkl})}{R}=\frac{2}{n-1}(L_{ijkl}-R_{ijkl}).
\end{equation}
\end{thm}

\begin{proof}
Differentiating equation \ref{thm3} with respect to 't' and using theorem \ref{thm2} the result follows.
\end{proof}

\section{Extensions}
For modified Riemann Extensions, since the scalar curvature vanishes, the concircular curvature tensor is same as the Riemannian curvature tensor. Further the conharmonic curvature tensor is equal to the conformal curvature tensor. \\
Ricci flow on modified Riemann extensions is the evolution of metric such that the class of metrics obtained under Ricci flow can be expressed as modified Riemann extensions of a base metric. 
We prove the following results for Ricci flow on modified Riemann extensions.

\begin{lem}
Laplacian of Ricci tensor is zero on modified Riemann extension.
\end{lem}

\begin{proof}
Laplacian of Ricci tensor is given by,\\
\begin{equation}\label{lap}
\triangle R_{ij}=g^{kl} R_{ij:k:l}.
\end{equation}
But, 
\begin{equation}
\begin{split}
g^{kl} R_{ij:k:l}=&g^{kl}R_{ij,k,l}-g^{kl}\Gamma^{\alpha}_{jk,l}R_{\alpha i}-g^{kl}\Gamma^{\alpha}_{jk}R_{\alpha i,l}-g^{kl}\Gamma^{\alpha}_{ik,l}R_{\alpha j}\\&-g^{kl}\Gamma^{\alpha}_{ik}R_{\alpha j,l}-g^{kl}\Gamma^{\alpha}_{il}R_{\alpha j,k}+g^{kl}\Gamma^{\alpha}_{il}\Gamma^{\beta}_{k\alpha}R_{\beta j}+g^{kl}\Gamma^{\alpha}_{il}\Gamma^{\beta}_{jk}R_{\beta \alpha}\\&-g^{kl}\Gamma^{\alpha}_{jl}R_{i\alpha,k}+g^{kl}\Gamma^{\alpha}_{jl}\Gamma^{\beta}_{\alpha k}R_{i\beta}+g^{kl}\Gamma^{\alpha}_{jl}\Gamma^{\beta}_{ik}R_{\beta \alpha}-g^{kl}\Gamma^{\alpha}_{kl}R_{ij,\alpha}\\&+g^{kl}\Gamma^{\alpha}_{kl}\Gamma^{\beta}_{i\alpha}R_{\beta j}+g^{kl}\Gamma^{\alpha}_{kl}\Gamma^{\beta}_{j\alpha}R_{i\beta}.
\end{split}
\end{equation}
From the properties of extended metric components we have,
$g^{kl}$ to be non zero atleast one of $k$ or $l$ must be greater than $n$. Suppose $k>n$, then $R_{ij,k}=0$. Also $R_{\alpha i}\neq 0$ only when $\alpha < n$ and $i<n$. But if $\alpha\leq n$ then $\Gamma^{\alpha}_{jk}=0$ since $k>n$. Similar argument makes all the terms on the right side of the equation to vanish. If $l>n$ then again $R_{ij,k,l}$ vanishes since $R_{ij,k}$ is a function of first $n$ coordinates. Also, since Christoffel symbols are preserved by extension, $\Gamma^{\alpha}_{jk,l}$ vanishes. Hence the result.
\end{proof}

\begin{thm}
The Ricci curvature tensor is independent of time for Ricci flow on modified Riemann  extensions.
\end{thm}

\begin{proof}
Let $M$ be an $n$-dimensional manifold.
The rate of change of Ricci tensor is given by
\begin{equation}
\frac{\partial R_{ik}}{\partial t}=\triangle R_{ik}+2g^{pr}g^{qs}R_{piqk}R_{rs}-2g^{pq}R_{pi}R_{qk}.
\end{equation}
for $i,k$ greater than $n$, $R_{ik}=0$. It is sufficient to prove for $i,k$ ranging from $1$ to $n$. For $g^{pr}$ and $g^{qs}$ to be non zero, either $p>n$ or $r>n$ and $q>n$ or $s>n$. Suppose $p>n$ and $q>n$. Then as discussed earlier $R_{piqk}=0$. If $s>n$ or $r>n$ then $R_{rs}=0$. Thus $2g^{pr}g^{qs}R_{piqk}R_{rs}=0$. Now  $g^{pq}$ is non zero for $p>n$ or $q>n$. But if $p>n$, $R_{pi}=0$ and similarly if $q>n$, $R_{qk}=0$. Hence the result.
\end{proof}

It must be noted here that the flow is not on the base manifold but on the extended space. We have proved the necessary and sufficient conditions for Modified Riemann extension under Ricci flow to stay as modified Riemann extensions. \\\\
We can restate the result in terms of metric. 

\begin{thm}
Ricci flow on modified Riemann Extensions is linear.
\end{thm}

\begin{proof}
Under Ricci flow on modified Riemann Extensions, the Ricci tensor is time invariant. Hence on solving \eqref{ric}
we get
\begin{equation}
g_{jk}(t)=R_{jk}t+g_{jk}(0).
\end{equation}
Thus the metric is linearly variying with time.
\end{proof}

\begin{exm}
modified Riemann extension of Schwarzchild metric has vanishing Ricci tensor and hence remains a trivial example.
\end{exm}

\begin{exm}
\rm{
Consider the hyperbolic metric $ds^2=\frac{1}{y^2}dx^2+\frac{1}{y^2}dy^2$. The modified Riemannian extension of this is 
\begin{equation}
ds^2=\frac{-4P}{y}dx^2-\frac{8P}{y}dxdy+\frac{4Q}{y}dy^2+2dxdp+2dQdy.
\end{equation}
where $c_{ij}=0$.\\
 Then $R_{11}=\frac{2}{y^2}=R_{22}$  and rest of the components equal to zero. Thus $g_{11}=\frac{2}{y^2}t+\frac{-4Q}{y}$ and $g_{22}=\frac{-4Q}{y}+\frac{2}{y^2}t$ with rest of the components independent of time which are the required class of  metric components.}
\end{exm}

\begin{thm}
For modified Riemann Extensions under Ricci flow, the rate of change of extended components of Weyl conformal tensor is the same as the rate of change of extended components of Riemann curvature tensor.
\end{thm}

\begin{proof}
For the extended space, the Weyl conformal tensor is given by
\begin{equation}
W_{ijkl}=R_{ijkl}-\frac{1}{n-2}(g_{ik}R_{jl}-g_{il}R_{jk}-g_{jk}R_{il}+g_{jl}R_{ik}).
\end{equation}
Differentiating partially with respect to  't' we get,
\begin{equation}
\begin{split}
\frac{\partial W_{ijkl}}{\partial t}=\frac{\partial R_{ijkl}}{\partial t}-&\frac{1}{n-2}(\frac{\partial g_{ik}}{\partial t}R_{jl}+g_{ik}\frac{\partial R_{jl}}{\partial t}-\frac{\partial g_{il}}{\partial t}R_{jk}-g_{il}\frac{\partial R_{jk}}{\partial t}\\&-\frac{\partial g_{jk}}{\partial t}R_{il}-g_{jk}\frac{\partial R_{il}}{\partial t}+\frac{g_{jl}}{\partial t}R_{ik}+g_{jl}\frac{\partial R_{ik}}{\partial t}).
\end{split}
\end{equation}
Using previous theorem and \eqref{ric} and rearranging we get
\begin{equation}
\frac{\partial W_{ijkl}}{\partial t}=\frac{\partial R_{ijkl}}{\partial t}-\frac{4}{n-2}(R_{il}R_{jk}-R_{ik}R_{jl}).
\end{equation}
Here again for any two of $i,j,k,l$ greater than $n$ the Ricci components are zero. In particular for all of them greater than $n$, we get the above result.
\end{proof}

\section*{Conclusion}
We have found the necessary and sufficient conditions for the the modified Riemann extension under Ricci flow  evolving to obtain a class of metrics which again are modified Riemann extensions.

\end{document}